\documentclass{article}
\usepackage[utf8]{inputenc}

  \usepackage{amssymb,amsthm,amsmath,mathrsfs,mathtools}
  \usepackage{booktabs}
  \usepackage{multirow}
  \usepackage{siunitx}
  \usepackage{amsfonts}
  \usepackage{enumitem}
  \usepackage{graphicx,placeins}
  \usepackage{lineno}
  \usepackage{hyperref}
  \hypersetup{
    colorlinks=true,
    linkcolor=blue,
    filecolor=blue,      
    urlcolor=blue,
    }
    
      \theoremstyle{plain}
  \newtheorem{theorem}{Theorem}
  \newtheorem{lemma}{Lemma}

  \theoremstyle{definition}
  \newtheorem{definition}{Definition}

\usepackage{algpseudocode,algorithm,algorithmicx}

\algrenewcommand\algorithmicrequire{\textbf{Input:}}
\algrenewcommand\algorithmicensure{\textbf{Output:}}
\algdef{SE}[DOWHILE]{Do}{doWhile}{\algorithmicdo}[1]{\algorithmicwhile\ #1}%

\usepackage[backend=biber, isbn=true, giveninits=true,maxbibnames=99,date=year,url=false, sorting=none,style=nature]{biblatex}
\renewbibmacro*{doi+eprint+url}{%
	\printfield{doi}%
	\newunit\newblock%
	\iftoggle{bbx:eprint}{%
		\usebibmacro{eprint}%
	}{}%
	\newunit\newblock%
	\iffieldundef{doi}{%
		\iffieldundef{eprint}{
			\usebibmacro{url+urldate}}%
		{}%
	}{}}
\usepackage{geometry}
\title{Using Spectral Submanifolds for Optimal Mode Selection in Model Reduction}
\author{Gergely Buza, Shobhit Jain\footnote{Corresponding author (shjain@ethz.ch)}, George Haller}
\date{\vspace{-5ex}
}

\begin{document}

\maketitle
\begin{center}
	Institute for Mechanical Systems, ETH Zürich \\
	Leonhardstrasse 21, 8092 Zürich, Switzerland 
	\par\end{center}
\begin{abstract}
Model reduction of large nonlinear systems often involves the projection of the governing equations onto linear subspaces spanned by carefully-selected modes. 
The criteria to select the modes relevant for reduction are usually problem-specific and heuristic. In this work, we propose a rigorous mode-selection criterion based on the recent theory of Spectral Submanifolds (SSM), which facilitates a reliable projection of the governing nonlinear equations onto modal subspaces. SSMs are exact invariant manifolds in the phase space that act as nonlinear continuations of linear normal modes.
Our criterion identifies critical linear normal modes whose associated SSMs have locally the largest curvature. These modes should then be included in any projection-based model reduction as they are the most sensitive to nonlinearities. To make this mode selection automatic, we develop explicit formulas for the scalar curvature of an SSM  and provide an open-source numerical implementation of our mode-selection procedure. We illustrate the power of this procedure by accurately reproducing the forced-response curves on three examples of varying complexity, including high-dimensional finite element models.

\end{abstract}

\section{Introduction}

The invariance of modal subspaces in linear oscillatory systems allows for a rigorous model reduction via linear projection onto any select group of linear normal modes~\cite{besselink2013}. For nonlinear systems, however, there are no mathematical results confirming the relevance of linear projection due to the general lack of invariance of modal subspaces. Indeed, a model reduction principle can only be justified mathematically if the reduced model is defined on an attracting invariant set of the nonlinear system \cite{HallerSFD}. Nonetheless, linear projection methods are routinely employed in the context of structural dynamics due to their simple implementation (cf.~\cite{Kuran1996,Ferhatoglu2018}, see \cite{benner2015} for a general survey). 

In practice, the accuracy of such a reduction procedure is dependent on an ad hoc choice of modes and hence needs to be verified on a case-by-case basis. A relevant example is an initially-straight, nonlinear von K\'{a}rm\'{a}n beam \cite{LAZARUS201235,MURAVYOV20031513,TOUZE2014,JAIN2018195}, where the axial and transverse degrees-of-freedoms are coupled only by the nonlinearities.
Refs.~\cite{LAZARUS201235,MURAVYOV20031513} 
propose a selection of modes supported by the physical understanding that a subset of axial modes should be included in the projection basis to account for the nonlinear bending-stretching coupling. Indeed, this reduction happens to result in an exact model reduction due to the presence of a slow manifold in this example, as shown in~\cite{JAIN2018195}. However, such physical intuition of selecting relevant axial modes is already unavailable upon a simple change in the geometry of the structure such as making the beam initially curved. More generally, heuristic mode-selection criteria are expected to be increasingly inaccurate as the size and the complexity of the underlying system increases.

Recent trends in nonlinear model reduction tackle these conceptual issues by constructing reduced-order models (ROMs) using invariant manifolds \cite{SHAW199385,Haller2016,HallerSFD,JainThesis}. While the computational feasibility of such invariant manifolds for high-dimensional dynamical systems is a subject of ongoing research, their relevance for nonlinear model reduction is certainly more appealing in comparison to linear projection. In particular, the spectral submanifolds (SSMs) \cite{Haller2016} allow the reduction of the nonlinear dynamics into an exact, lower-dimensional invariant manifold in the phase space. This SSM attracts all neighboring solutions, which ensures exponentially fast synchronization of general oscillations with their reduced model. The accuracy of the model can be made arbitrarily high without increasing its dimension: one can simply compute higher-order terms in a Taylor expansion for the SSM. 

In this work, we leverage the theoretical relevance of SSMs to select a smaller set of modes optimally for the purposes of reduction by modal projection. We perform this selection by computing the local curvature of the relevant SSM in the modal directions. These directional curvatures highlight the modes that would affect the nonlinear response most significantly. Starting with an initial set of modes using linear mode superposition, we develop a procedure to identify a linear subspace that captures the local curvature of the relevant SSM. We automate this process so that the user obtains an optimal set of modes with minimal input.

After describing the basic setup, we review the essential elements of SSM theory and its numerical implementation in Section~\ref{sect:SSMcomputation}. We then introduce the geometric notions behind our proposed nonlinear mode selection along with a motivational example in Section \ref{sect:idea}. The notion of the directional scalar curvature of an SSM is developed in Section~\ref{sect:curvature}. Finally, in Section \ref{sect:examples}, we use the reduced-order models~(ROM) generated from our directional-curvature-based mode selection criterion to accurately reproduce the forced response curves in finite-element examples.

\section{Setup}
\label{sect:setup}
In this work, we focus on periodically forced mechanical systems of the form
\begin{equation}
M \ddot{q} + C \dot{q} + K q + S (q, \dot{q}) = \varepsilon f (\Omega t), \qquad 0< \varepsilon \ll 1, 
\label{eq:forcedsystem}
\end{equation}
where $q(t) \in \mathbb{R}^n$ is the vector of generalized coordinates; $M \in \mathbb{R}^{n\times n}$ is the positive definite mass matrix; $K \in \mathbb{R}^{n\times n}$ is the positive semi-definite stiffness matrix; $C \in \mathbb{R}^{n\times n}$ is the damping matrix which is assumed to satisfy the proportional damping hypothesis, i.e., $C=\alpha K+ \beta M$ for some  ${\alpha, \beta\in \mathbb{R}}$; $S(q,\dot{q}) = \mathcal{O}\left( \left\vert q \right\vert^2,\left\vert q  \right\vert \left\vert \dot{q} \right\vert,\left\vert \dot{q} \right\vert^2 \right)$ is the nonlinearity which assumed to be of class $ C^r $ in its arguments for some integer $ r\ge 1 $; and $f$ is a  $T$-periodic forcing function ($ T = 2 \pi / \Omega $), with an amplitude parameter $ \varepsilon > 0$.

The proportional damping hypothesis enables us to simultaneously diagonalize the linear part of system~\eqref{eq:forcedsystem} using the undamped eigenmodes, $ u_j \in \mathbb{R}^n $, defined as
\begin{align}
Ku_j = \omega_j^2 Mu_j, \quad j = 1,\dots,n.
\end{align}
Without any loss of generality, we assume that the eigenmodes are mass-normalized, i.e.,  \begin{equation}\label{massnorm}
\left\langle u_i,  M u_j \right\rangle = \delta^i_j,
\end{equation} 
where~$ \delta^i_j $ denotes the Kronecker delta. We use the linear transformation $q = U \mu$, where $\mu \in \mathbb{R}^n $ denotes the vector of modal coordinates, and ${U = [u_1, \dots , u_n] \in \mathbb{R}^{n\times n}}$  is the transformation matrix composed of the eigenmodes of the undamped system, to express system~\eqref{eq:forcedsystem} in modal coordinates as 
\begin{align}
\ddot{\mu}_i  + 2 \zeta_i \omega_{i} \dot{\mu}_i + \omega_{i}^2 \mu_i + s_i(\mu,\dot{\mu}) = \varphi_i (t), \qquad i \in {1,\dots,n},
\label{modal}
\end{align}
where $s_i(\mu,\dot{\mu}):= \left\langle u_i,  S \left(U\mu,U\dot{\mu}\right) \right\rangle$, $ \zeta_i :=\frac{1}{2\omega_i}\left\langle u_i,  C u_i \right\rangle  $, and $ \varphi_i (t):= \left\langle u_i, \varepsilon f (\Omega t) \right\rangle $.

We then separate the nonlinear system~\eqref{modal} into two subsystems,
\begin{align}\label{master1}
    &\ddot{\xi}_i  + 2 \zeta_i \omega_{i} \dot{\xi}_i + \omega_{i}^2 \xi_i + s_i\left((\xi,\eta),(\dot{\xi},\dot{\eta})\right) = \varphi_i (t), \qquad i \in I,
    \\ 
    &\ddot{\eta}_j + 2 \zeta_j \omega_{j} \dot{\eta}_j + \omega_{j}^2 \eta_j + s_j\left((\xi,\eta),(\dot{\xi},\dot{\eta})\right) = \varphi_j (t), \qquad j \in J,
    \label{enslaved}
\end{align}
where system~\eqref{master1} is composed of a set of master modes~$I\subset \{1,\dots,n\}$, with the modal coordinates denoted by $\xi$; and system~\eqref{enslaved} is composed of the enslaved modes $J:=\{1,\dots,n\}\backslash I$, with modal coordinates denoted by $\eta$. We denote by $m$ the cardinality of $I$, i.e., the number of master modes in the eq.~\eqref{master1}.

The main principle behind any projection-based model reduction technique lies in suitably identifying the set $ I $ of master modes in a way, so that the ROM 
\begin{equation}\label{master}
\ddot{\xi}_i  + 2 \zeta_i \omega_{i} \dot{\xi}_i + \omega_{i}^2 \xi_i + s_i\left((\xi,0),(\dot{\xi},0)\right) = \varphi_i (t), \qquad i \in I,
\end{equation}
gives a reasonably accurate approximation to the true evolution of the master modes in eq.~\eqref{master1}.
 In the following, we develop a procedure to automate the selection of master modes using SSM theory, which we review next.

\section{Spectral submanifolds}
\label{sect:SSMcomputation}
Using the notation 
\begin{align}
z = \left(\begin{array}{c}
q \\
\dot{q} \end{array}\right),
\quad B = \left(\begin{array}{cc}
0 & \mathrm{I}_{n\times n}\\
-M^{-1}K & -M^{-1}C
\end{array}\right),
\quad F(z) = \left(\begin{array}{c}
0 \\
-M^{-1}S(q,\dot{q}) \end{array}\right).
\end{align}
we rewrite system~\eqref{eq:forcedsystem} for $ \varepsilon = 0 $ as the first-order, autonomous system
\begin{equation}
\dot{z} = B z + F(z), 
\label{eq:mechanicalsystem}
\end{equation}
whose linearization at $ z=0 $ is given by
\begin{equation}
    \dot{z} = B z.
    \label{linODE}
\end{equation}
For each mode $ i $ of the second-order system~\eqref{modal}, we denote the corresponding pair of eigenvalues of the first-order system~\eqref{linODE} by
\begin{equation}\label{key}
\lambda_{2i-1},\lambda_{2i} = \left(-\zeta_i \pm \sqrt{\zeta_i^2 - 1}\right) \omega_i, \quad i = 1,\dots, n.
\end{equation}
Hence, for any distinct eigenvalue pair $ \lambda_{2i-1},\lambda_{2i} $ associated to mode $ i $, we obtain a two-dimensional invariant subspace~$E_i$ of system~\eqref{linODE}. By linearity, we can generate higher-dimensional invariant subspaces of system~\eqref{linODE} by direct-summing such two-dimensional subspaces. A \textit{spectral subspace}~\cite{Haller2016} is a general invariant subspace of this type. For instance, the spectral subspace $ E_I $ generated by the set $ I $  of master modes is given as
\begin{equation}\label{mastermodalsubspace}
E_I := \bigoplus_{i \in I} E_i,
\end{equation} 
where $ \bigoplus $ is the direct-sum operator.

A \emph{spectral submanifold} (SSM) \cite{Haller2016} is an invariant manifold of system~\eqref{eq:mechanicalsystem} that serves as the smoothest nonlinear continuation of the spectral subspace of the linearized system~\eqref{linODE}. Specifically, the SSM emanating from spectral subspace~$ E_I $, is defined as follows. 
\begin{definition}
    An SSM, $M_I$, corresponding to a spectral subspace $E_I$ of the operator $B$ is an invariant manifold of the nonlinear system \eqref{eq:mechanicalsystem} such that
    \begin{enumerate}
        \item $M_I$ is tangent to $E_I$ at the origin and $\mathrm{dim}(M_I)=\mathrm{dim}(E_I)=2m$; 
        \item $M_I$ perturbs smoothly from $E_I$ under the addition of nonlinear terms; 
        \item $M_I$ is strictly smoother than any other invariant manifold satisfying (i) and (ii).
    \end{enumerate}
\end{definition}

The existence and uniqueness of such SSMs is guaranteed by the following theorem:

\begin{theorem} (Haller \& Ponsioen~\cite{Haller2016}, Theorem 3)
	\label{thm:ssmexistence}
Assume that 
 	\begin{enumerate}
 	\item the relative spectral quotient $ \sigma(E_I) := \mathrm{Int} \left( \frac{\min_{k \in J} \mathrm{Re} \lambda_k}{\max_{i \in I} \mathrm{Re} \lambda_i} \right) $ satisfies $ \sigma(E_I)\le r $,
 	\item the following low-order nonresonance conditions hold:
 	\begin{equation}
 	\sum_{i \in I} m_i  \lambda_i \neq \lambda_k, \qquad k \in J; \qquad 2 \leq \sum_{i \in I} m_i \leq \sigma(E_I), \qquad m_i \in \mathbb{N}. \label{eq:resonancecond}
 	\end{equation}
 \end{enumerate} 
 Then:
	\begin{enumerate}
		\item There exists a class-$C^{r}$ SSM, $M_I$, for system \eqref{master1}-\eqref{enslaved} that is unique amongst all $C^{\sigma (E_I)+1}$ manifolds that are tangent to $E_I$ at $(\xi, \dot{\xi},\eta,\dot{\eta})=0$.
		\item $M_I$ can locally be viewed as the image of an open set $O \subset E_I$ under the embedding  $$\psi: O \rightarrow \mathbb{R}^{2n}.$$
	\end{enumerate}
\end{theorem}
So far we have discussed SSMs for the autonomous ($ \varepsilon = 0 $) limit of system~\eqref{eq:mechanicalsystem}. Similarly, however, SSMs can also be defined in the non-autonomous ($ \epsilon>0 $) setting. In that case, for $ \varepsilon > 0 $ small enough, the role of the fixed point at $ z =0 $ is taken over by a small-amplitude periodic orbit $ \gamma_{\varepsilon} $ created by the periodic forcing. This periodic orbit will have SSMs emerging from its spectral subbundles that are direct products of the periodic orbit with spectral subspaces of the origin. An SSM is then a fibre bundle that perturbs smoothly from a vector bundle $ \gamma_{\varepsilon} \times E_I $ under the addition of the nonlinear terms, as long as appropriate resonance conditions stated in Theorem 4 in \cite{Haller2016} hold. 
The fibers of the forced SSM inherit their topological properties and leading-order shape from the unforced setting, for small enough $\varepsilon$ (cf. Breunung \& Haller~\cite{breunung2017}). Hence, we intend to use the autonomous SSM, $ M_I $, in determining the influence of nonlinearity on the near-equilibrium forced response. 

\subsection{SSM computation}

Theorem~\ref{thm:ssmexistence} allows us to approximate the SSM, $M_I$, around the origin as a graph $\eta(x)$ over the subspace $E_I$ via a Taylor expansion, i.e.,
\begin{equation}
\eta_k (x) = \left\langle x , W_k x \right\rangle + \mathcal{O} (\vert x \vert ^3) , \qquad k \in J,
\label{eq:etak}
\end{equation}
where \begin{equation}
 x = \left(\xi, \dot{\xi}\right),
\end{equation}  
and $W_k =W_k^{T}  \in \mathbb{R}^{2m\times 2m}$ denote the matrix of SSM coefficients to be determined. As we will show, the local curvature of $M_I$ provides a robust criterion for mode selection and the second-order coefficients $ W_k $  are sufficient to compute for this purpose. 

The explicit solutions for the coefficients $W_k$ of the SSM can be found from a direct invariance computation, as detailed in Appendix~\ref{appendixD}. The main equations we solve are of the form
\begin{equation}
B_k \cdot W_k = -R_k, \qquad k \in J, 
\label{Wieqtensor}
\end{equation}
where the $R_k = R_k^{T}  \in \mathbb{R}^{2m\times 2m}$ are the quadratic coefficients extracted from the nonlinearities as
\begin{equation}
s_k (x) = \left\langle x, R_k x \right\rangle + \mathcal{O} (\vert x \vert ^3), \qquad k \in J,
\label{eq:Rk}
\end{equation}
and  $B_k$ is a fourth-order tensor for each $ k \in J $, whose entries are given by
\begin{equation}
B_{k,st}^{rq} = 2 A_s^r A_t^q + A_m^q A_t^m \delta_s^r + A_m^r A_s^m \delta_t^q + 2 \zeta_k \omega_{k} \left( A_t^q \delta_s^r + A_s^r \delta_t^q \right) +  \omega_{k}^2 \delta_r^s \delta_t^q. 
\label{Bistrq}
\end{equation}
Here we have followed the Einstein summation convention; the upper index is the row index and the lower index is the column index of a matrix, and \begin{equation}\label{A}
 A = \begin{pmatrix}
0 & \mathrm{I}_{m\times m} \\
- K_{I} & -C_{I}
\end{pmatrix}, \qquad K_I = \mathrm{diag}\left(\{\omega_i^2\}_{i\in I}\right),\quad C_I = \mathrm{diag}\left(\{2\zeta_i\omega_i\}_{i\in I}\right).
\end{equation}

Veraszt\'{o} et al.~\cite{VERASZTO2020115039} have already developed matrix equations to determine $W_k$ (see eqs.~(E.7-E.9) in~\cite{VERASZTO2020115039}), but their expressions were less amenable to numerical implementation. The expressions developed have been implemented in open-source MATLAB scripts~\cite{SST}. 

\section{Mode selection and directional curvature of SSM}
\label{sect:idea}
\subsection{Initial mode selection based on modal superposition}
The linearization of system~\eqref{eq:forcedsystem}, 
\begin{equation}
M \ddot{q}(t) + C \dot{q}(t) + K q(t)  = \varepsilon f (\Omega t), 
\label{eq:linearforcedsystem}
\end{equation}
exhibits a unique periodic response at the same frequency $ \Omega $, as that of forcing $ f $~\cite{Jain2019c}. Thanks to modal superposition, this linearized periodic response can be accurately approximated using a small set of eigenmodes along which the system is excited by the forcing $ f $. Furthermore, due to its hyperbolicity, the periodic response of the linearized system~\eqref{eq:linearforcedsystem} would be a valid approximation to the nonlinear periodic response of system~\eqref{eq:forcedsystem} for small enough $ \varepsilon > 0 $~\cite{Guckenheimer1983a}. Hence, any projection-based ROM~\eqref{master} must include the modes essential for reproducing the linearized response in the master mode set $ I $.

\subsection{Updating the set of modes based on directional curvatures of SSM}
Let $ I \subset \{1,\dots,n\} $ be the minimal set indexing the modes required for reproducing the periodic response of the linearized system~\eqref{eq:linearforcedsystem} using modal superposition. The SSM, $ M_I $, describes how the dominant spectral subspace, $ E_I $, of the linear system~\eqref{eq:linearforcedsystem} deforms locally, upon inclusion of nonlinear terms from system~\eqref{eq:forcedsystem} in the limit of $ \varepsilon \to 0 $. 

For a projection-based ROM to be effective, the master mode set $ I $ must be updated to capture this deformation by appending further modes to the projection subspace $ E_I $, if necessary. To assess the leading-order deformation of the SSM over directions spanned by linear normal modes, we use the \textit{directional curvatures} of  $ M_I $. Based on these directional curvatures, we will \textit{update} the mode set $ I $ to efficiently approximate the nonlinear response of system~\eqref{eq:forcedsystem} via the ROM~\eqref{master}. Figure \ref{fig:ssm} gives a geometric sketch of this idea.

\begin{figure}[h]
\includegraphics[width=0.7\textwidth]{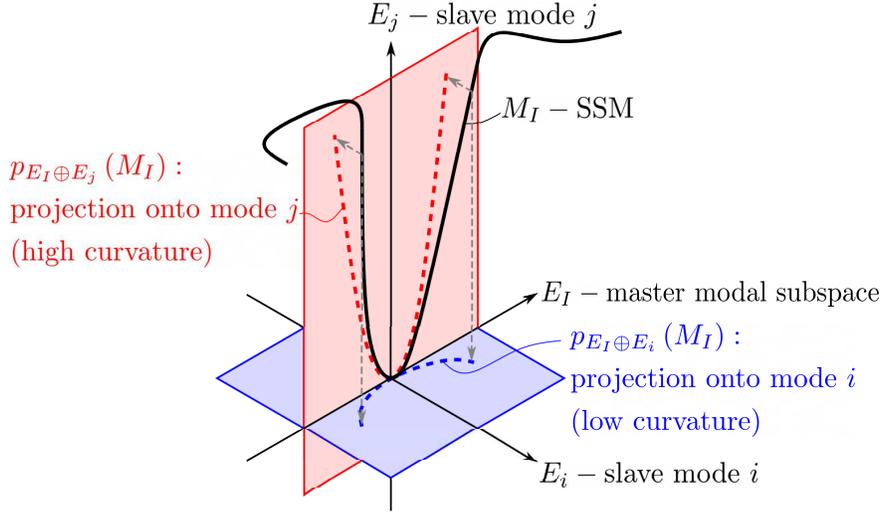}
\centering
\caption{Geometric interpretation of our mode selection criterion: Computing the SSM, $M_I$, for an initial mode set of $I$, then projecting it onto modes $i$ and $j$ yields the $2m$ dimensional manifolds $p_{E_I \oplus E_i} \left( M_I \right)$ and $p_{E_I \oplus E_j} \left( M_I \right)$ displayed by blue and red dashed lines, respectively. 
The high curvature of $p_{E_I \oplus E_j} \left( M_I \right)$ relative to that of $p_{E_I \oplus E_i} \left( M_I \right)$ prompts us to include mode $j$ in the master mode set $ I $, i.e., let $ I \to I \cup \{j\} $, which gives us the updated projection subspace, $ E_I \oplus E_j $, shaded in red. }
\label{fig:ssm}
\end{figure}
Let $p_E: \mathbb{R}^{2n} \rightarrow E$ define the orthogonal projection from the full phase space~$\mathbb{R}^{2n}$ onto a spectral subspace,~$E$. Then ${p_E \left( M_I  \right) \subset E}$ is a manifold of the same dimension as the SSM, $M_I$, provided that $E$ is selected such that $E_I \subset E$.
Therefore, projecting $ M_I $ orthogonally onto $ E_I \oplus E_k $ for each $ k \in J $, we obtain $n-m$ new manifolds, one for each enslaved mode, of dimension $2m$, in the form   
\begin{equation}
p_{E_I \oplus E_k} \left( M_I  \right) \subset E_I \oplus E_k, \qquad k \in J.
\end{equation}
The curvature of $p_{E_I \oplus E_k} \left( M_I \right)$ provides us with a notion of directional curvature for the SSM, $ M_I $, in the direction of mode $k$, as shown in Figure \ref{fig:ssm}.

Geometric intuition arising from Figure~\ref{fig:ssm} suggests that enslaved modes which have the largest directional curvatures, such as  mode $ j $, must be included in a projection subspace used for obtaining the ROM~\eqref{master}. This is because the projection subspace~${E_I \oplus E_j}$ (dashed red curve in Figure~\ref{fig:ssm}) effectively captures the local deformation of $ M_I $ in comparison to the projection subspace~${E_I \oplus E_i}$ (dashed blue curve in Figure~\ref{fig:ssm}). Hence, mode $ j $ is influential in determining the near-equilibrium nonlinear response and we propose that the master mode set should be updated to include such modes,  i.e., should be enlarged as $ I \to I \cup \{j\} $.

Before making this simple idea algorithmically precise, we motivate it with a small physical example in the following section.
\FloatBarrier
\subsection{Motivating example for mode selection}
\label{sect:motivexample}

To motivate our proposed mode selection procedure, we consider a single-mass spring system (see~Figure~\ref{fig:motivexample}) that is a three-dimensional variant of a similar example considered by Touz\'{e} et al.~\cite{TOUZE200477} and Breunung \& Haller~\cite{breunung2017}. 
\begin{figure}[h]
\includegraphics[width=0.24\textwidth]{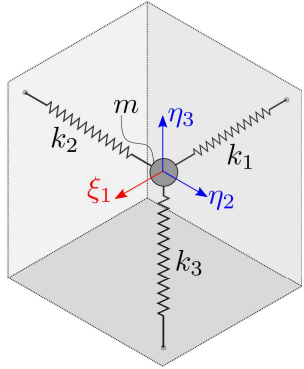}
\centering
\caption{The single-mass spring system.}
\label{fig:motivexample}
\end{figure}
Using the same potential functional (see eq.~(14) in~\cite{TOUZE200477}), we obtain the equations of motion
{\small
\begin{align}
    &\ddot{\xi}_1 +2\zeta_1 \omega_{1} \dot{\xi}_1 + \omega_{1}^2 \xi_1 + \frac{\omega_{1}^2}{2} \left( 3 \xi_1^2+ \eta_2^2+ \eta_3^2 \right) + \omega_{2}^2 \xi_1 \eta_2 +\omega_{3}^2 \xi_1 \eta_3 + \frac{\omega_{1}^2+\omega_{2}^2+\omega_{3}^2}{2} \xi_1 \left(  \xi_1^2+ \eta_2^2+ \eta_3^2 \right) = f_1 (t), \label{eq:motivsystem1} \\
    &\ddot{\eta}_2 + 2\zeta_2 \omega_{2} \dot{\eta}_2  + \omega_{2}^2 \eta_2  + \frac{\omega_{2}^2}{2} \left( 3 \eta_2^2+ \xi_1^2+ \eta_3^2 \right) + \omega_{1}^2  \eta_2 \xi_1 +\omega_{3}^2  \eta_2 \eta_3 + \frac{\omega_{1}^2+\omega_{2}^2+\omega_{3}^2}{2} \eta_2 \left(  \xi_1^2+ \eta_2^2+ \eta_3^2 \right) = f_2 (t), \label{eq:motivsystem2}\\
    & \ddot{\eta}_3 + 2\zeta_3 \omega_{3} \dot{\eta}_3  + \omega_{3}^2 \eta_3  + \frac{\omega_{3}^2}{2} \left( 3 \eta_3^2+ \xi_1^2+ \eta_2^2 \right) + \omega_{1}^2  \eta_3 \xi_1 +\omega_{2}^2  \eta_3 \eta_2 + \frac{\omega_{1}^2+\omega_{2}^2+\omega_{3}^2}{2} \eta_3 \left(  \xi_1^2+ \eta_2^2+ \eta_3^2 \right) =f_3 (t),
    \label{eq:motivsystem3}
\end{align}}
where $\omega_{i} = \sqrt{k_i/m}$ are the eigenfrequencies and the $\zeta_i$ are the viscous damping coefficients.

We choose $\omega_1 = 2$, $\omega_2 = 3$, $\omega_3 = 5$, $\zeta_1 = 0.01$, $\zeta_2 = 0.02$ and $\zeta_3 = 0.08$ as model parameters.  We apply a harmonic forcing along the first mode, i.e. $f_1 = F \cos \left( \Omega t \right)$ with an amplitude of  $F = 0.02$ and leave the remaining modes unforced, i.e., let $ f_2 = f_3 = 0 $. Hence, due to modal superposition, only the first mode participates in the linearized response as all other modes are left unforced. We are interested in obtaining the forced response curves around the first natural frequency $ \omega_1 $ of the system. We compute the periodic response for the forcing frequency  $ \Omega $ values in the interval $ [0.7\omega_1 $,  $ 1.3\omega_1 $].

Although there are advanced, example-specific mode selection recipes in literature for ROMs~\cite{Amabili2013}, perhaps the most straight forward and general strategy is to simply use a number of low-frequency modes that comfortably span the forcing spectrum~\cite{Geradin}.  Accordingly, a two-mode ROM would contain the modes 1 and 2 in a projection basis. We refer to this mode set as $ I_1 = \{1,2\}$.

To apply our proposed mode-selection criterion, we choose $I_0=\{ 1 \}$ as the initial master mode set, motivated by the linearized response. Calculating the second-order coefficients of $ M_{I_0} $, one readily observes from eqs.~\eqref{eq:motivsystem2}-\eqref{eq:motivsystem3} that
\begin{equation}
R_k = \begin{pmatrix}
 \omega_k^2/2 & 0 \\
  0 & 0
\end{pmatrix}, \qquad k \in \{ 2,3 \},
\end{equation}
since $\eta_k = \mathcal{O} \left( \vert \xi_1 \vert^2 \right)$. The coefficient matrices $W_2$ and $W_3$, obtained from the solution of the invariance equation~\eqref{Wieqtensor} at second order are
\begin{equation}
W_2 =
\begin{pmatrix}
    0.0712  &  0.0008 \\
    0.0008  & -0.1428
\end{pmatrix}, \qquad
W_3 = 
\begin{pmatrix}
   -0.8778  &  0.1033 \\
    0.1033  &  0.0953
\end{pmatrix}.
\end{equation}
Without a formal notion of scalar curvature at this point, we may treat $\Vert W_k \Vert_2$ as a measure of the curvature of the SSM in the direction of slave mode $k$. Comparing $\Vert W_2 \Vert_2=0.1428$ with $\Vert W_3 \Vert_2=0.8886$, we deduce that mode 3 has a significantly higher directional curvature in comparison to mode 2. Our proposed criterion would update the master mode set as $ I_0 \to I_0 \cup \{3\} $ to  include modes 1 and 3 in a two-mode projection basis. We refer to this mode set as $ I_2 = \{1,3\} $.

We compare the ROMs obtained from the two mode sets $I_1$ and $I_2$ with the full solution by computing the forced response curves with the results shown in Figure~\ref{fig:motiv}. The two ROMs result in remarkably different responses: the hardening response for $ I_1 $ and softening response for $ I_2 $. As Figure~\ref{fig:motiv} shows, the ROM obtained from the mode set~$ I_2 $, based on the directional curvatures of the SSM, correctly predicts the response and establishes the relative importance of mode $3$ over mode $ 2 $. 
 
\begin{figure}[h]
\includegraphics[width=0.5\textwidth]{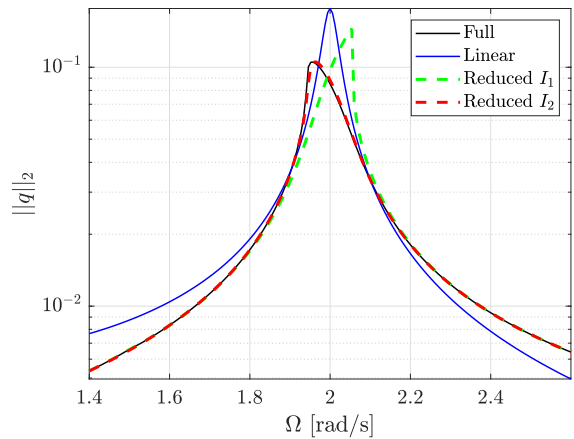}
\centering
\caption{Response curves obtained with projections onto two linear spectral subspaces $E_{I_1}$ and $E_{I_2}$ for two index sets $I_1=\left\{ 1,2 \right\}$ and $I_2=\left\{ 1,3 \right\}$. Response curves obtained using the full system and the linearized system are also shown for reference. The vector $q=(\xi_1, \eta_2, \eta_3)$ denotes the set of generalized coordinates.}
\label{fig:motiv}
\end{figure}

While the use of $ \Vert W_k \Vert_2 $ as a scalar measure of directional curvature seems intuitive in this simple example, we need a mathematical notion of directional curvature for arbitrary, finite-dimensional SSMs to make our mode-selection criterion systematic. We will introduce such a directional curvature next in Section~\ref{sect:curvature}.

\FloatBarrier
\section{Scalar curvature of an SSM and automated mode selection}
\label{sect:curvature}
We require a scalar quantity that is representative of the curvature of multi-dimensional manifolds. For our purposes, the classic sectional curvature~\cite{lee1997riemannian} loses its applicability beyond two dimensions, since it depends on the choice of a two-dimensional plane in the tangent space of the manifold. Therefore, we use an extension called the \textit{scalar curvature}, which is obtained by taking the trace of the Ricci tensor of the manifold~\cite{lee1997riemannian}. In the following, we develop explicit formulas for the scalar curvature of $ M_I $ along any given modal direction.


We first establish that our SSM approximation is a Riemannian manifold $(M_I,g)$ with an appropriate metric $g$ by expressing it as a graph over the master modal subspace $ E_I $. Let $\varphi: E_I \rightarrow \mathbb{R}^{2(n-m)}$ denote the approximating function developed for the enslaved modes $(\eta,\dot{\eta})$:
\begin{equation}
\varphi_k \left( x \right) :=
\begin{cases}
        \left\langle x , W_{k} x \right\rangle, & k = 1, \ldots , n-m,\\
        2 \left\langle x , W_{k-n+m} A x \right\rangle, & k = n-m+1, \ldots , 2(n-m),\\
\end{cases}
\label{eq:phi}
\end{equation}
where we have rearranged the indices of $W_k$ such that $W_k:= W_{J(k)}$, where $J$ is the set of enslaved modes.
Our approximate manifold for the SSM is then given as
$$
M_I = \mathrm{graph} (\varphi) =  \big\{ (x, \varphi \left( x \right) ) \, \vert \, x \in E_I \big\}.
$$
Let us denote by $\psi : E_I \supset O \rightarrow \mathbb{R}^{2n}$ the $C^{\infty}$ embedding defined on an open neighbourhood $O$ of $0 \in E_I$, given by
\begin{equation}
    \psi (x) : = (x, \varphi \left( x \right) ).
    \label{eq:psi}
\end{equation}

Define vector fields $e_1,\ldots,e_{2m}$ along $\psi$ by
\begin{equation}
    e_i (x) : = \frac{\partial \psi}{ \partial x^i} (x) \in T_{\psi(x)}M_I.
\end{equation}
Then the components $g_{ij}:O \rightarrow \mathbb{R}$ of the standard metric $g$ on $M_I$ inherited from the Euclidean space $\mathbb{R}^{2n}$ are given by
\begin{equation}
   g_{ij} = \left\langle e_i, e_j \right\rangle, 
   \label{metric}
\end{equation}
where $\langle \cdot , \cdot \rangle$ is the Euclidean inner product.
The vector fields $e_i$ form a local frame on $\psi(O)$, and at $\psi(0)$ they form an orthonormal basis of $T_{\psi(0)}M_I$, since with the given embedding \eqref{eq:psi}, we have that $g_{ij}(0) = \delta_i^j$.

We now proceed towards the definition of the scalar curvature of $M_I$.
For this, we require the Ricci tensor, which is a 2-tensor whose coefficients are given as\footnote{We are using the Einstein summation convention here and throughout the entirety of this section for upper-lower pairs of indices. We include the summation signs for clarity where the convention fails to work.}
\begin{equation}
    r_{ij} :=  R^k_{kij},
    \label{eq:Riccicomponents}
\end{equation}
where the $R_{ijk}^l $ are the coefficients of the Riemann curvature tensor, given explicitly in terms of the Christoffel symbols $\Gamma^k_{ij}$ as
\begin{equation}
  R_{ijk}^l = \partial_i \Gamma_{jk}^l - \partial_j \Gamma_{ik}^l +  \Gamma_{is}^l \Gamma^s_{jk} - \Gamma^l_{js} \Gamma^s_{ik},
  \label{eq:curvaturecomponents}
\end{equation}
with
\begin{equation}
    \Gamma^k_{ij} = \frac{1}{2} g^{kl} \big( \partial_i g_{jl} + \partial_j g_{il} - \partial_l g_{ij} \big),
    \label{eq:christoffel}
\end{equation}
where $\partial_i$ is shorthand notation for $\partial / \partial x^i$ (see \cite{petersen2006riemannian} or \cite{salamon2020LECTURENOTES}, for instance).

These preliminaries allow us to define the scalar curvature of a Riemannian manifold (see \cite{lee1997riemannian}, for instance) as follows.

\begin{definition}
    The scalar curvature of a Riemannian manifold $(M,g)$ is the function given by
    \begin{equation}
        \mathrm{scal}_g := g^{ij} r_{ij}.
        \label{eq:scalgdef}
    \end{equation}
\end{definition}

In the following, we denote by $\mathrm{curv}(I)$ the scalar curvature of the manifold $M_I$ evaluated at the origin.

\begin{lemma}
\label{LEMMA1}
Let $M_I$ denote the autonomous SSM corresponding to a master mode set $I$.
Then the scalar curvature of $M_I$ at the origin is given by
\begin{equation}
    \mathrm{curv}(I) = \frac{1}{2} \sum_{a,b=1}^{2m}  \big( - \partial_a \partial_a g_{bb} + 2 \partial_a \partial_b g_{ab} - \partial_b \partial_b g_{aa} \big) (0).
    \label{eq:scalcurvature}
\end{equation}
\end{lemma}
\begin{proof}
This is a special case of a more general statement given on page 128 of \cite{lee1997riemannian}.
We provide a direct proof in Appendix \ref{appendixLEMMA1}.
\end{proof}

We can now make sense of the directional curvature discussed in Section \ref{sect:idea} (see~Figure~\ref{fig:ssm}).
We may replace the embedding $\psi$ given in \eqref{eq:psi} with $\tilde{\psi}_k(x) = (x, \tilde{\varphi}_k(x))$, where we declare that
\begin{equation}
    \tilde{\varphi}_k (x) = 
    \begin{pmatrix}
    \left\langle x, W_k x \right\rangle \\
    2 \left\langle x, W_k A x \right\rangle
    \end{pmatrix}.
    \label{eq:tildephi}
\end{equation}
Then the corresponding manifold $\mathrm{graph} (\tilde{\varphi}_k)$ is the projection of the full SSM to the $2m+2$ dimensional spectral subspace $E_I \oplus E_k$, i.e.,
\begin{equation}
\mathrm{graph} (\tilde{\varphi}_k) = p_{E_I \oplus E_k} \left( M_I \right).
\label{eq:GRAPHtildephi}
\end{equation}
Computing the scalar curvature of $\mathrm{graph} (\tilde{\varphi}_k)$ gives us the desired formula for the directional curvature of the SSM, $ M_I $, along mode $ k $. In the following lemma, we provide ready-to-use formulas for the scalar curvature of $ M_I $, $ \mathrm{curv}(I), $ and for its directional curvature along any enslaved mode $ k $, which we denote by $ \mathrm{curv}_k(I) $.

\begin{lemma}
\label{scalarcurvaturelemma}
Let $M_I$ denote the autonomous SSM corresponding to a master mode set $I$.
Then the scalar curvature of $M_I$ at the origin is given by
\begin{equation}
        \mathrm{curv}(I) = 4 \sum_{a,b=1}^{2m} \sum_{k=1}^{n-m}
     \left[
    W_{k,a}^a W_{k,b}^b + 4 \left(W^a_{k,r}A^r_a \right)  \left(W^b_{k,s}A^s_b \right) -
    (W^b_{k,a})^2 - (W_{k,r}^a A^r_b + W^b_{k,r} A^r_a)^2
    \right].
    \label{scalresult}
\end{equation}
Furthermore, the scalar curvature of the projected manifold $p_{E_I \oplus E_k} \left( M_I \right)$, i.e., the directional curvature of $ M_I $ along mode $ k $, at the origin is given as
\begin{equation}
    \mathrm{curv}_k(I) = 4 \sum_{a,b=1}^{2m}
     \left[
    W_{k,a}^a W_{k,b}^b + 4 \left(W^a_{k,r}A^r_a \right)  \left(W^b_{k,s}A^s_b \right) -
    (W^b_{k,a})^2 - (W_{k,r}^a A^r_b + W^b_{k,r} A^r_a)^2
    \right], \qquad k \in J.
    \label{scalkresult}
\end{equation}

\end{lemma}
\begin{proof}
The proof is carried out in Appendix \ref{appendixG}.
\end{proof}

Finally, the explicit formulas for directional curvature given in Lemma~\ref{scalarcurvaturelemma} allow us to devise an automated mode-selection procedure as follows:
\begin{enumerate}[label=\arabic*.]
\item Choose an initial master mode set $ I $ that approximates the linearized periodic response, i.e., utilize linear modal superposition. 

\item Compute the directional curvature $ \mathrm{curv}_k(I) $ of $ M_I $ (see eq.~\eqref{scalkresult}) along each slave mode $ k $ in the set $ J = \{1,\dots,n\} \backslash I$. 

\item Choose a minimal subset $ P \subset J $ of slave modes that captures the directional curvatures up to a user-defined tolerance $ 0<p\ll 1 $. Specifically, we require that
\begin{equation}
	 \frac{\sum_{k\in J} \vert \mathrm{curv}_k(I) \vert - \sum_{k\in P} \vert \mathrm{curv}_k(I) \vert}{\sum_{k\in J} \vert \mathrm{curv}_k(I) \vert } \le p 
	 \label{ptolerance}
\end{equation}

\item Update the master mode set as $ I \to I \cup P $.
\end{enumerate}

Additionally, if the user a priori specifies a desired number of modes $N$ in the ROM, then the steps~2-4 of the selection procedure may be repeated until the mode set reaches cardinality $N$. In practice, the set $ P $ is robust with respect to the tolerance $ p $ and in the authors' experience,  $ p $ values in the range $ 0.05 $ to $ 0.15 $ provide optimal output.

A pseudo code of the automation algorithm is given in Appendix \ref{appendixalgorithm}. A numerical implementation of this algorithm is downloadable in the form of MATLAB scripts~\cite{SST}.

\FloatBarrier
\section{Numerical examples}
\label{sect:examples}

\subsection{Straight von Kármán beam}
\label{vonkarmansection}

As our first example, we use an initially straight von Kármán beam (see e.g.,~\cite{JAIN2018195}). Due to the bending-stretching nonlinear coupling, high-frequency axial modes are required in this problem to approximate the full nonlinear response using ROMs (see~\cite{LAZARUS201235,MURAVYOV20031513}).

A finite element discretization using cubic shape functions in the transverse direction and linear shape functions in the axial direction leads to the general form given by eq.~\eqref{eq:forcedsystem}. We choose a linear viscoelastic damping model, resulting in the proportional damping matrix 
\begin{equation}
C = \frac{\kappa}{E} K,
\end{equation}
where $E$ is the Young's modulus and $\kappa$ is the material damping coefficient. Due to this choice of damping,  $S$ in eq.~\eqref{eq:forcedsystem} is a purely position-dependent, cubic nonlinear function given by
\begin{equation}
S_k (q) = a_k^{ij} q_i q_j + b_k^{ijl} q_i q_j q_l,
    \label{samplenonlin}
\end{equation}
where the coefficients $a_k^{ij}$ and $b_k^{ijl}$ are polynomial stiffness coefficients. 

We consider an aluminium beam, which we divide into 10 elements of equal size.
The model parameters are $E = 70$ GPa, $\kappa = 0.1$ GPa$\cdot$s, $\rho = 2700$ kg/m$^3$ with geometric parameters $l = 1$ m (length), $h = 1$ mm (height) and $b=0.1$ m (width). 
We choose doubly clamped boundary conditions, i.e., both axial and transverse displacements are constrained at both ends. We apply a uniform-in-space, periodic-in-time external load in the transverse direction with an excitation frequency of $\Omega = 26$ rad/s, which is between the eigenfrequencies of the first and second modes. For the forcing amplitude, we take $F=2.3$ N.

To obtain a ROM via the proposed mode-selection procedure, we first choose an initial mode set $I_0 = \{1,2,3,4,5  \}$, which accurately recovers the linearized periodic response.  We then compute the directional scalar curvatures of the SSM, $M_{I_0}$, by first extracting the quadratic term $W_k$ in each direction $k \in J$ via eq.~\eqref{Wieqtensor}, and then substituting in the explicit formulas~\eqref{scalkresult}.

\begin{figure}[h]
\includegraphics[width=0.5\textwidth]{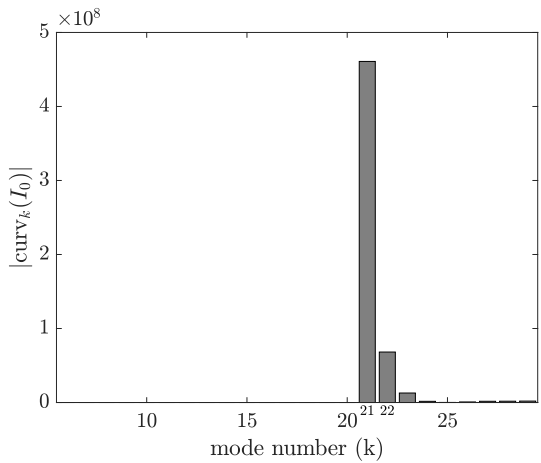}
\centering
\caption{The directional curvatures of $ M_{I_0} $ for an initial master mode set $I_0 = \{1,2,3,4,5  \}$ along each of the slave modes of an initially straight von K\'{a}rm\'{a}n beam. The axial modes 21 and 22 display the most prominent directional curvatures, confirming one's physical intuition.}
\label{fig:placeholder}
\end{figure}

We plot the directional scalar curvatures along each mode in Figure~\ref{fig:placeholder}. As the figure shows, our method automatically recommends that the axial modes 21 and 22 should be included in the ROM (with the user-defined tolerance chosen as $p=0.05$, cf.\ eq.~\eqref{ptolerance}), which confirms the available physical intuition based on the nonlinear bending-stretching coupling.
 \begin{table}[h]
	\centering
	\sisetup{table-format=6.4}
	\begin{tabular}{ll} \toprule
		
		{\textbf{Modes in ROM~\eqref{master}}} &  {\textbf{Relative error} $ e_r $~\eqref{error}}\\ \toprule
		$I_0=\left\{ 1,2,3,4,5 \right\}$ & 0.22 \\
		$I_1=\left\{ 1,2,3,4,5,6,7,8,9,10 \right\}$ & 0.18 \\
		$I_2=\left\{ 1,2,3,4,5,21,22\right\}$ &  0.03\\ \bottomrule
	\end{tabular}
	\caption{The relative error obtained from different ROMs of an initially straight von K\'{a}rm\'{a}n beam.  Note that the high-frequency axial modes in the mode set $ I_2 $, obtained from our automated mode selection procedure, are crucial for accurately approximating the nonlinear response (see~Figure~\ref{fig:u4t}).}
	\label{tab:error}
\end{table}

We use a mass-weighted relative error norm to compare the accuracy of ROMs to the full solution as 
\begin{equation}\label{error}
e_r = \frac{\Vert q_r - q\Vert_M}{\Vert q\Vert_M},
\end{equation}
where $ q(t) $ denotes the full solution, $ q_r(t) $ denotes the reduced solution, and $ \Vert \bullet \Vert_M $ denotes the mass norm defined as
\begin{equation}\label{mass_norm}
 \Vert x \Vert_M = \sqrt{\int_{[0,T]} \left\langle x(t),Mx(t) \right\rangle \mathrm{d} t},
\end{equation}
where $ T $ is the minimal time period of the periodic response.
  
Table~\ref{tab:error} compares the relative error $ e_r $ of ROMs based on three different mode sets $ I_0, I_1, I_2 $, where $ I_0 $ is the mode set that accurately approximates the linearized periodic response using modal superposition; $ I_1 $ is the set of the 10 lowest frequency modes used for comparison purposes; and $ I_2 $ is the mode set obtained from the proposed mode-selection procedure. We plot the periodic response of axial and transverse displacements at the 4th node of the beam ($ 0.3l $ from the constrained end) as a function of time in Figure~\ref{fig:u4t}. Clearly, the axial movement of the beam is not captured by the ROMs $ I_0, I_1 $ as they only contain low-frequency bending modes, an issue that is automatically rectified by the proposed mode-selection procedure using the mode set~$I_2$.
\begin{figure}[h]
\includegraphics[width=1\textwidth]{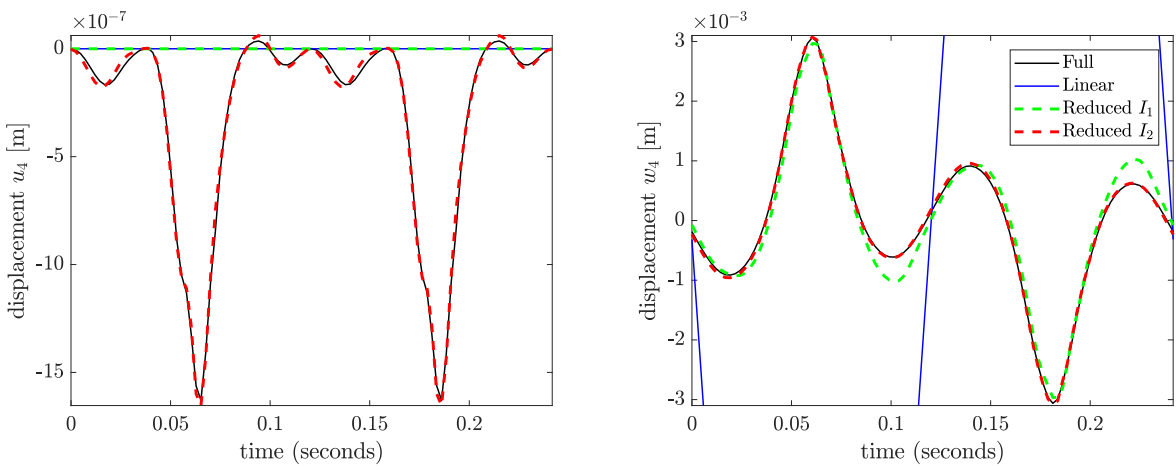}
\centering
\caption{Axial displacement (left) and transverse displacement (right) at the 4th node of a von K\'{a}rm\'{a}n beam, forced periodically with a loading amplitude $F=2.3$ N and an excitation frequency $\Omega=26$ rad/s. We compare the periodic response across ROMs (see eq.~\eqref{master}) obtained from mode sets $ I_1 = \{1,\dots,10\} $ and $ I_2 = \{1,\dots,5,21,22\} $, along with the linearized and full responses.}
\label{fig:u4t}
\end{figure}

To verify our predictions across a range of forcing amplitudes, we compute the periodic response for increasing forcing amplitude at the same forcing frequency $ \Omega = 26 $ rad/s, as shown in Figure~\ref{fig:incrementalamp}. We note that the mode set $ I_2 $ produces consistently good approximation for large forcing amplitudes, where the ROM obtained from the heuristically chosen low-frequency modes $ I_1 $ diverges from the full solution branch.
Furthermore, upon increasing the forcing amplitudes, we observe a nonphysical response for projection-based ROMs using the mode set $ I_1 $, whereas the solution corresponding to $I_2$ remains numerically stable. We detail this phenomenon for forcing amplitude $F=2.44$ N in Appendix~\ref{appendixB}. 
\begin{figure}[h]
\includegraphics[width=0.5\textwidth]{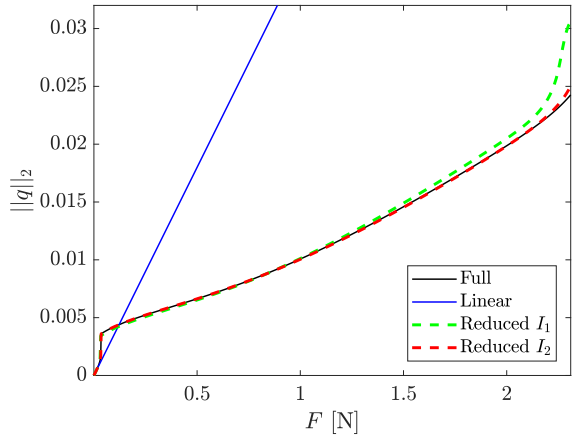}
\centering
\caption{Steady-state response curves as functions of increasing forcing amplitudes at the excitation frequency $\Omega = 26$ rad/s. The jump on the left-hand side of the figure corresponds to a jump between two separate solutions, one of which only exists for lower amplitudes, while the other one exists for larger amplitudes.}
\label{fig:incrementalamp}
\end{figure}

\FloatBarrier
\subsection{Curved von Kármán beam}
\label{section:curvedbeam}

As a second example, we consider a curved beam in the form of a circular arch, such that its midpoint is raised by $a = 5$ mm relative to its ends. We use the same geometrical and material parameters as defined in Section~\ref{vonkarmansection}, except for the beam height, which is chosen as $h=7$ mm. We apply time-periodic forcing on all transverse degrees of freedom with an amplitude of $F=80$ N.

A distinguishing aspect of this example is that the curved geometry introduces a linear coupling between the axial and transverse degrees of freedom of the beam. Hence, one can no longer identify any axial or transverse modes for mode selection based on the same physical intuition that we utilized for the initially straight beam (see Section~\ref{vonkarmansection}).

\begin{figure}[h]
\includegraphics[width=0.5\textwidth]{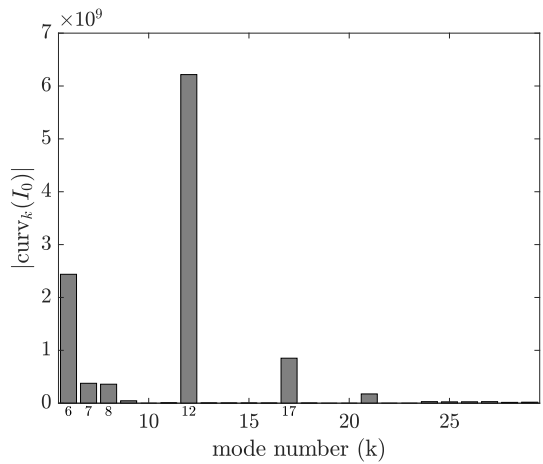}
\centering
\caption{The directional curvatures of $ M_{I_0} $ for an initial master mode set $I_0 = \{1,\dots,5  \}$ along each of the slave modes of a curved von K\'{a}rm\'{a}n beam. }
\label{fig:curvedrecomm}
\end{figure}

Once again, we choose an initial mode set $I_0 =  \{ 1,\dots,5 \}$ that accurately reproduces the linearized response. The directional curvatures of the SSM, $ M_{I_0} $, are shown in Figure~\ref{fig:curvedrecomm}. Performing our mode selection procedure with $p=0.05$ once more, we find that the master mode subset should be updated by including slave modes $ \{ 6,7,8,12,17 \}$. We denote by $I_2 = \left\{ 1,\dots,8,12,17 \right\}$ the updated mode set obtained with the proposed selection procedure.
\begin{figure}[h]
	\includegraphics[width=0.5\textwidth]{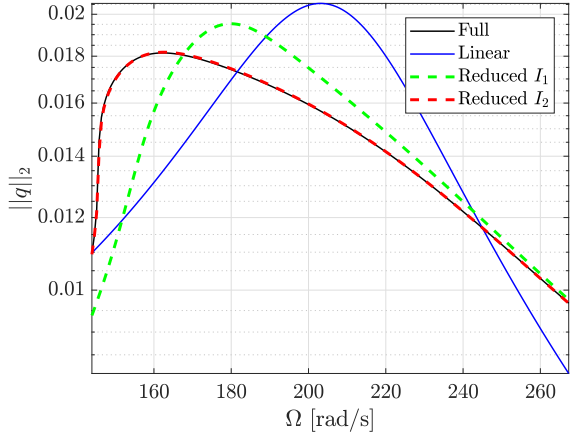}
	\centering
	\caption{A comparison of ROMs obtained from master mode sets $I_1=\left\{ 1,\dots,10 \right\}$ and $I_2=\left\{ 1,\dots,8,12,17 \right\}$ (see eq.~\eqref{master}), along with the full and linearized solutions. The frequency range shown is based on the first undamped eigenfrequency of the system, which is $208$ rad/s.}
	\label{fig:curvedcompar}
\end{figure}

We now compute the periodic response of the ROM~\eqref{master}, where the master mode set $ I = I_2 $. We compare our results to a ROM composed of a master mode set of the same size but heuristically comprising the lowest frequency modes as $I_1=\left\{ 1,\dots,10 \right\}$. The response curves are shown in Figure \ref{fig:curvedcompar}. Note that we generally expect a softening type behaviour for large enough forcing amplitudes in the case of curved beams~\cite{beamexplanation}. As Figure~\ref{fig:curvedcompar} shows, our proposed mode selection procedure (mode set $ I_2 $) systematically produces a reliable prediction of the steady-state response compared to a heuristic choice ($ I_1 $) of modes.

\FloatBarrier
\section{Conclusions}
\label{sect:conclusions}
We have developed a systematic procedure to obtain an optimal set of modes for projection-based reduced-order modeling of nonlinear mechanical systems. This nonlinear mode selection procedure relies on the directional curvatures of the spectral submanifolds (SSMs) constructed around the dominant modal subspaces. These SSMs form the centerpieces of near-equilibrium, nonlinear steady-state response and facilitate an exact model reduction of the nonlinear response.

While the SSMs can also be directly used to approximate the steady-state response~\cite{PONSIOEN2018269,Ponsioen2019ExactMR,breunung2017} in nonlinear mechanical systems, their computational feasibility for realistic high-dimensional problems is a subject of ongoing research. Our method relies on SSM theory, but still employs the widely applied linear projection to obtain a reduced-order model, which is straightforward to implement and whose computational advantages are well-understood.

We have shown through two beam examples that our mode selection criterion not only confirms the physical intuition of selecting axial modes to capture the nonlinear bending-stretching coupling, but also provides accurate results when such case-specific intuition is not available. The proposed nonlinear mode selection procedure, whose pseudo-code is given in Algorithm~\ref{appendixalgorithm}, is openly available in the form of open-source MATLAB scripts~\cite{SST}.

\appendix
\section{Convergence issues for the von Kármán beam example}
\label{appendixB}

In this Appendix, we compare the same error estimates as in Section \ref{vonkarmansection}, but for a larger forcing amplitude $F=2.44$N. We obtain solutions for such large amplitudes by  sequential continuation, i.e., by incrementally increasing the forcing amplitude in steps and using the solution from the previous step as an initial solution for the current step.

Table \ref{tab:appendixb} shows the error values as in Section \ref{vonkarmansection}, while Figure \ref{fig:appendixb} shows the comparison of axial and transverse displacements for ROMs obtained using mode sets ${I_1}$ and ${I_2}$. From the figure, we observe convergence to a nonphysical response for transverse displacements using the mode set $ I_1 $, which contains heuristically chosen low-frequency modes. On the other hand, our proposed mode selection procedure still provides a reliable approximation to the steady state response using mode set $ I_2 $.

  \begin{table}[h]
 \centering
  \sisetup{table-format=6.4}
    \begin{tabular}{ll} \toprule
    	
		{\textbf{Modes in ROM~\eqref{master}}} &  {\textbf{Relative error} $ e_r $~\eqref{error}}\\ \toprule
  $I_0=\left\{ 1,\dots,5 \right\}$ &  15.84 \\ 
  $I_1=\left\{ 1,\dots,10 \right\}$  & 15.25 \\ 
   $I_2=\left\{ 1,\dots,5,21,22\right\}$  & 0.11 \\ \bottomrule 
    \end{tabular}
	\caption{The relative error obtained from different ROMs of an initially straight von K\'{a}rm\'{a}n beam periodically forced with a forcing amplitude $ F = 2.44 $ N and a forcing frequency of $ \Omega = 26 $ rad/s (cf.~Figure~\ref{fig:appendixb}).}
  \label{tab:appendixb}
 \end{table}

\begin{figure}[h]
\includegraphics[width=1\textwidth]{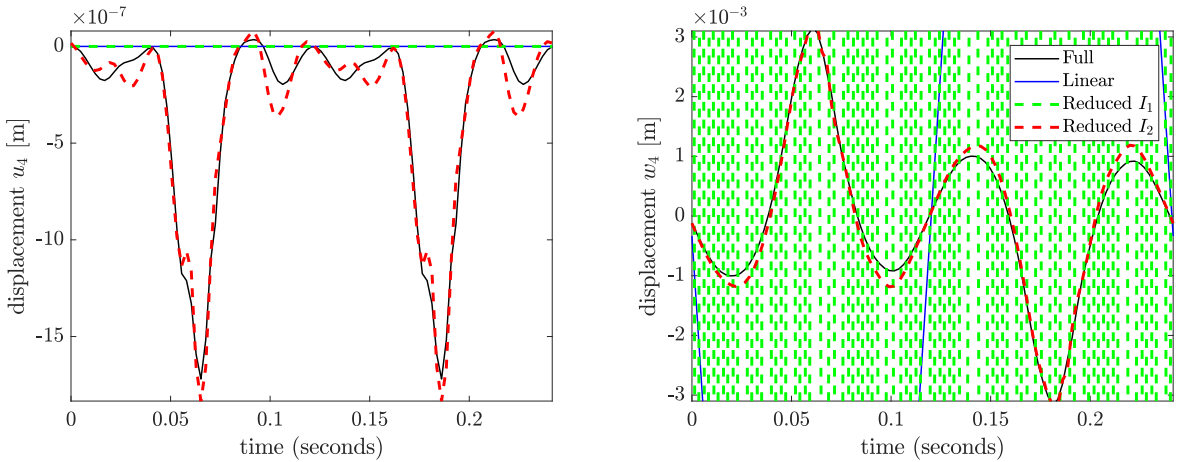}
\centering
\caption{Axial displacement (left) and transverse displacement (right) at the 4th node of a von K\'{a}rm\'{a}n beam, forced periodically with a loading amplitude $F=2.44$ N and an excitation frequency $\Omega=26$ rad/s. We compare the periodic response across ROMs (see eq.~\eqref{master}) obtained from mode sets $ I_1 = \{1,\dots,10\} $ and $ I_2 = \{1,\dots,5,21,22\} $, along with the linearized and full responses.}
\label{fig:appendixb}
\end{figure}

\FloatBarrier
\section{Derivation of the leading-order coefficients of the SSM}
\label{appendixD}

In this Appendix, we use a direct invariance computation to obtain expressions for the leading-order coefficients $W_k$ of the SSM as in \eqref{eq:etak}.
We begin by computing the first and second derivatives of $\eta_k$:
\begin{align}
&\dot{\eta}_k (x) = 2 \left\langle x , W_k A x \right\rangle+\mathcal{O} (\vert x \vert ^3) , & k \in J, \label{deta}\\
&\ddot{\eta}_k (x) = 2 \left\langle x , (A^T W_k A + W_k A^2) x \right\rangle+\mathcal{O} (\vert x \vert ^3) ,  & k \in J. \label{ddeta}
\end{align}
We also need to extract the quadratic nonlinearities from $s_k$,
\begin{equation}
s_k (x) = \left\langle x, R_k x \right\rangle + \mathcal{O} (\vert x \vert ^3), \qquad k \in J,
\end{equation}
where we have used that $\eta (x)= \mathcal{O} (\vert x \vert ^2)$.
Again, for uniqueness, we require that the $R_k$ are symmetric.
We can now substitute the derivatives \eqref{deta} and \eqref{ddeta} into equation \eqref{enslaved} to obtain
\begin{equation}
\left\langle x , \left[ 2(A^T W_k A + W_k A^2) + 4 \zeta_k \omega_{k} W_k A + \omega_{k}^2 W_k + R_k \right] x \right\rangle+\mathcal{O} (\vert x \vert ^3) = 0, \qquad k \in J.
\end{equation}
From this we can deduce that, for our second order approximation, the symmetric part of the quadratic coefficient matrix should vanish for all $k$, which leads to the following set of linear equations:
\begin{equation}
 2 A^T W_k A + W_k A^2 + \left( A^2 \right)^T W_k + 2 \zeta_k \omega_{k} ( W_k A + A^T W_k ) + \omega_{k}^2 W_k + R_k = 0, \qquad k \in J.   
 \label{Wiequation}
\end{equation}
From here onwards summation is implied over repeated indices to simplify the notation.
We can write equation \eqref{Wiequation} in a simpler form if we introduce the fourth-order tensor
\begin{equation}
B_k = B_{k,st}^{rq}  e^s \otimes e^t \otimes e_r \otimes e_q,
\end{equation}
with
\begin{equation}
    B_{k,st}^{rq} = 2 A_s^r A_t^q + A_m^q A_t^m \delta_s^r + A_m^r A_s^m \delta_t^q + 2 \zeta_k \omega_{k} \left( A_t^q \delta_s^r + A_s^r \delta_t^q \right) +  \omega_{k}^2 \delta_r^s \delta_t^q,
    \label{Bistrqa}
\end{equation}
where the upper index is the row index and the lower index is the column index of a matrix, and $\delta^i_j$ is the Kronecker delta. System \eqref{Wiequation} can now be written as
\begin{equation}
    B_k \cdot W_k = -R_k, \qquad k \in J, 
    \label{Wieqtensorap}
\end{equation}
where dot denotes the inner product of tensors, given by
\begin{equation}
B_k \cdot W_k = B_{k,st}^{rq} W_{rq}  e^s \otimes e^t.
\end{equation}
One can simply vectorize equation \eqref{Wieqtensorap} by rearranging the entries of $W_k$ and $R_k$ into vectors, which leads to a $4m^2$-dimensional linear system of equations, readily solvable via a matrix inversion.
For a more systematic approach to the higher-order computation of the SSM that is based on similar methodology, we refer the reader to Ponsioen et al.~\cite{PONSIOEN2018269,Ponsioen2019ExactMR}.

\section{Proof of Lemma \ref{LEMMA1}}
\label{appendixLEMMA1}

\begin{proof}
Combining equations \eqref{eq:Riccicomponents} and \eqref{eq:scalgdef}, we get that
\begin{align}
    \mathrm{curv}(I) &= g^{ib} R^a_{aib} (0) \nonumber\\
    &= \sum_{b=1}^{2m} R^a_{abb} (0),
    \label{eq:appCcurv}
\end{align}
where the second equality used that $g_{ib}(0) = g^{ib} (0) = \delta_b^i$.
Before advancing any further, we note that the $\Gamma^k_{ij}$ vanish at $0$ in our setting.
This can be inferred from the general fact that Christoffel symbols vanish at $p \in M$ for a chart that induces an orthonormal basis of $T_pM$.
Alternatively, a substitution of the specific embedding $\psi$ (from \eqref{eq:psi}) and the metric \eqref{metric} into the formulas for $\Gamma_{ij}^k$ gives the same result.

We may now proceed by inserting the formula for the $R^l_{ijk}$ \eqref{eq:curvaturecomponents} into \eqref{eq:appCcurv}, which yields
\begin{align}
    \mathrm{curv}(I) &= \sum_{b=1}^{2m} \Big(\partial_a \Gamma_{bb}^a - \partial_b \Gamma_{ab}^a + \Gamma_{as}^a \Gamma^s_{bb} - \Gamma^a_{bs} \Gamma^s_{ab} \Big) (0) \nonumber \\
    &= \sum_{b=1}^{2m} \Big(\partial_a \Gamma_{bb}^a - \partial_b \Gamma_{ab}^a \Big) (0).
\end{align}
Expanding the first half of this expression according to \eqref{eq:christoffel} gives
\begin{align}
\sum_{b=1}^{2m} \partial_a \Gamma^a_{bb}(0) &=\frac12 \sum_{b=1}^{2m}  \left[  \partial_a g^{al} \big( 2 \partial_b g_{bl} - \partial_l g_{bb}  \big) +g^{al} \big( 2 \partial_a \partial_b g_{bl} - \partial_a \partial_l g_{bb} \big)\right] (0) \nonumber\\
&=\frac12 \sum_{b=1}^{2m}  \left[ \left( \partial_a g^{al} \right) \left( 2 g_{li} \Gamma_{bb}^i \right) 
 +g^{al} \big( 2 \partial_a \partial_b g_{bl} - \partial_a \partial_l g_{bb} \big)\right] (0) \nonumber\\
 &=\frac12\sum_{b=1}^{2m}  \delta_a^l \left( 2 \partial_a \partial_b g_{bl} - \partial_a \partial_l g_{bb} \right) (0) \nonumber\\
 &=\frac12\sum_{a,b=1}^{2m}  \big( 2 \partial_a \partial_b g_{ba} - \partial_a \partial_a g_{bb} \big) (0).
\end{align}
A similar computation for the second term yields
\begin{equation}
  \sum_{b=1}^{2m}  \partial_b \Gamma_{ab}^a (0) =\frac12 \sum_{a,b=1}^{2m}    \partial_b \partial_b g_{aa}   (0),
\end{equation}
which proves \eqref{eq:scalcurvature}.
\end{proof}


\section{Proof of Lemma \ref{scalarcurvaturelemma}}
\label{appendixG}

\begin{proof}
Throughout this proof we do \textit{not} use the Einstein summation convention on the indices $a$ and $b$, but we use it on every other index.
We only need to compute the terms appearing in \eqref{eq:scalcurvature}.
We start by computing $\partial_a \partial_a g_{bb}$. 
First, we write $\psi$ out explicitly using its definition \eqref{eq:psi} and the definition of $\varphi$:
\begin{equation}
   \psi (x) = \left(x, \varphi(x) \right) =
\begin{pmatrix}
x \\
\left\langle x, W_1 x \right\rangle \\
\vdots \\
2 \left\langle x, W_{n-m} A x \right\rangle
\end{pmatrix} 
\in \mathbb{R}^{2n}.
\label{appc5}
\end{equation}
In the following we will make use of 
\begin{equation}
    \frac{\partial}{\partial x^a} \left\langle x, W_k x \right\rangle = 2 W_{k,i}^a x^i,
    \label{appc1deriv}
\end{equation}
and
\begin{equation}
    \frac{\partial}{\partial x^a} 2 \left\langle x, W_k A x \right\rangle= 2 \sum_{i=1}^{2m} \left(W_{k,j}^i A^j_a + W_{k,l}^a A^l_i  \right) x^i.
    \label{appc2deriv}
\end{equation}
We can now compute the smooth function $g_{bb}$ via the definition of the metric \eqref{metric}, using \eqref{appc5}, \eqref{appc1deriv} and \eqref{appc2deriv}:
\begin{align}
    g_{bb} (x) &= \left\langle \frac{\partial \psi}{ \partial x^b}, \frac{\partial \psi}{ \partial x^b} \right\rangle (x) \nonumber\\
    &= 1 + 4 \sum_{k=1}^{n-m} \left( W_{k,i}^b x^i \right)^2 + 4 \sum_{k=1}^{n-m} \left[ \sum_{i=1}^{2m} \left( W_{k,j}^i A^j_b +  W_{k,l}^b A^l_i \right) x^i \right]^2.
    \label{eq:gbb}
\end{align}
Differentiating the above expression with respect to $x^a$ once yields
\begin{equation}
\partial_a g_{bb} (x) = 8 \sum_{k=1}^{n-m} \left( W_{k,i}^b x^i \right) W_{k,a}^b
+ 8 \sum_{k=1}^{n-m} \left[ \sum_{i=1}^{2m} \left( W_{k,j}^i A^j_b +  W_{k,l}^b A^l_i \right) x^i \right] 
\left( W_{k,j}^a A^j_b +  W_{k,l}^b A^l_a \right),
\end{equation}
proceeding once more we obtain
\begin{equation}
    \partial_a \partial_a g_{bb} (x) = 8 \sum_{k=1}^{n-m} \left[ \left( W_{k,a}^b \right)^2 
+  \left( W_{k,j}^a A^j_b +  W_{k,l}^b A^l_a \right)^2 \right].
\label{appc2}
\end{equation}
In particular, $\partial_a \partial_a g_{bb}$ is a constant function for a second order approximation of the SSM.
Note that even if we would take a higher order approximation, evaluating the above function at $0$ would yield the same.
Furthermore, \eqref{appc2} is symmetric in $a$ and $b$ (by symmetry of $W_{k,a}^b$), and thus can be used for both terms $\partial_a \partial_a g_{bb}$ and $\partial_b \partial_b g_{aa}$ appearing in \eqref{eq:scalcurvature}.

As for the term $\partial_a \partial_b g_{ab}$, we proceed in a similar manner.
First we compute $g_{ab}$ via \eqref{appc1deriv} and \eqref{appc2deriv}:
\begin{equation}
    g_{ab} (x) = 4 \sum_{k=1}^{n-m} \left( W_{k,i}^a x^i \right) \left( W_{k,j}^b x^j \right)
    + 4 \sum_{k=1}^{n-m} \left[ \sum_{i=1}^{2m} \left( W_{k,j}^i A^j_a +  W_{k,l}^a A^l_i \right) x^i \right]
    \left[ \sum_{r=1}^{2m} \left( W_{k,q}^r A^q_b +  W_{k,s}^b A^s_r \right) x^r \right],
    \label{eq:gab}
\end{equation}
then differentiating with respect to $x^b$ once yields
\begin{align}
\partial_b  g_{ab} (x) &=  4 \sum_{k=1}^{n-m} W_{k,b}^a \left( W_{k,j}^b x^j \right) +
4 \sum_{k=1}^{n-m} W_{k,b}^b \left( W_{k,i}^a x^i \right) \nonumber\\
&+4 \sum_{k=1}^{n-m}  \left( W_{k,j}^b A^j_a +  W_{k,l}^a A^l_b \right) 
\left[ \sum_{r=1}^{2m} \left( W_{k,q}^r A^q_b +  W_{k,s}^b A^s_r \right) x^r \right] \nonumber\\
&+4 \sum_{k=1}^{n-m} \left[ \sum_{i=1}^{2m} \left( W_{k,j}^i A^j_a +  W_{k,l}^a A^l_i \right) x^i \right]
 \left( W_{k,q}^b A^q_b +  W_{k,s}^b A^s_b \right),
\end{align}
then differentiating with respect to $x^a$ yields
\begin{equation}
    \partial_a \partial_b  g_{ab} (x) =
4 \sum_{k=1}^{n-m} \left[ 
(W^b_{k,a})^2 + (W_{k,r}^a A^r_b + W^b_{k,r} A^r_a)^2
\right]+
4 \sum_{k=1}^{n-m} \left[ 
W_{k,a}^a W_{k,b}^b + 4 \left(W^a_{k,r}A^r_a \right)  \left(W^b_{k,s}A^s_b \right)
\right].
\label{egyemmeg2}
\end{equation}
Again, note that the same thing applies here as for \eqref{appc2}, namely, we have that even for a higher order approximation of the SSM evaluating \eqref{egyemmeg2} at $0$ would yield the same constant function.
Now substituting both \eqref{appc2} and \eqref{egyemmeg2} into \eqref{eq:scalcurvature} gives \eqref{scalresult}.

For the second statement, we compute the scalar curvature of the manifold $\mathrm{graph (\tilde{\varphi}_k)}$ with $\tilde{\varphi}_k$ given in \eqref{eq:tildephi}, as this is the same manifold as $p_{E_I \oplus E_k} \left( M_I \right)$ (c.f.\ \eqref{eq:GRAPHtildephi}).
Explicitly, this means that we repeat the same computation as given above, now with a different embedding $\tilde{\psi}_k (x) = (x, \tilde{\varphi}_k(x))$.
The only difference that arises with this change is that in the expressions for $g_{bb}$ and $g_{ab}$ (equations \eqref{eq:gbb} and \eqref{eq:gab}) the summations over $k$ are removed.
Consequently, the entire proof carries over with all summations over $k$ removed.
The result \eqref{scalkresult} now follows.

\end{proof}


\section{SSM-based mode selection algorithm}
\label{appendixalgorithm}

\begin{algorithm}
    \caption{Automation Algorithm (Implemented in SteadyStateTool~\cite{SST})}
    \begin{algorithmic}[1]
    \Require{Full system~\eqref{eq:forcedsystem} definition, coefficients $a_i^{ij}$ of the form \eqref{samplenonlin}, the tolerance $p$ defined in \eqref{ptolerance}, the maximum or desired number of modes $N$ in the ROM, and a boolean $type$ (\textbf{true} if the nonlinear mode selection should be repeated until the desired cardinality $ N $ of modes is reached, \textbf{false} otherwise )}
    \Ensure{An optimal mode set for $I$ projection-based ROM~\eqref{master}} 
	\Statex
    \Statex  \underline{Modal superposition}
    \State Compute the periodic response,  $x_{lin}$, of the full linearized  system~\eqref{eq:linearforcedsystem}
    \State $z \leftarrow U^T M x_{lin}$
    \State $n_i \leftarrow \Vert z_i \Vert_2, \quad \forall i$
    \While{$\#(I) < \mathrm{ceil}(2N/3)$} \Comment{Limiting the number of modes obtained based on linear analysis}
    \Statex \Comment{$\#(I)$ denotes the cardinality of the set $I$}
    \If{$\sum_{i \in I} n_i > 0.9 \sum_{i=1}^{n} n_i$} 
    \State \textbf{break}
    \EndIf
    \State $J \leftarrow \{1,\ldots,n\} \setminus I$
    \State $q \leftarrow \mathrm{arg} \max_{i \in J} n_i$
    \State $I \leftarrow I \cup \{q \}$
    \EndWhile
	\Statex
    \Statex \underline{Nonlinear mode selection} 
    \While{$\#(I) < N$} \Comment{Repeating nonlinear selection unless specified otherwise by  $type$}
    \State Compute the $R_k$ as in \eqref{eq:Rk} from the coefficients $a_k^{ij}$
    \State Compute the $W_k$ via \eqref{Wieqtensor}
    \State Compute the directional scalar curvatures $\mathrm{curv}_k(I)$ via \eqref{scalkresult}
    \State $J \leftarrow \{1,\ldots,n\} \setminus I$
    \State $g \leftarrow 0$
    \State $crit \leftarrow (1-p) \sum_{k \in J} \vert \mathrm{curv}_k(I)  \vert$
    \State $P \leftarrow \emptyset$ \Comment{$P$ is the recommended set of modes}
    \While{$g<crit$} \Comment{Selecting modes until we reach the scalar curvature criterion $p$}
    \State $g \leftarrow g+ \max_{k \in J} \vert \mathrm{curv}_k(I)  \vert$
    \State $J \leftarrow J \setminus \{ \mathrm{arg} \max_{k \in J} \vert \mathrm{curv}_k(I)  \vert \}$
    \State $P \leftarrow P \cup \{ \mathrm{arg} \max_{k \in J} \vert \mathrm{curv}_k(I) \vert \}$
    \EndWhile
    \If{$N- \left( \#(I)+\#(P) \right)<0$} \Comment{Ensure user-specified limit on maximum number of modes}
    \State $I \leftarrow  I \cup P(1:(N-\#(I))) $
    \Else
    \State $I \leftarrow   I \cup P $
    \EndIf
    \If{$type=$\textbf{false}} \Comment{Break the loop if type is false}
    \State \textbf{break}  
    \EndIf
    \EndWhile
    \end{algorithmic}
\end{algorithm}

\clearpage

\printbibliography[title={References}]

\end{document}